\documentclass[12pt,letterpaper]{article}
\usepackage[T1]{fontenc}
\usepackage[english]{babel}	
\usepackage{german}

\usepackage{graphicx}
\usepackage[a4paper]{geometry}
\usepackage{fullpage}
\usepackage{hyperref}

\usepackage{algorithm}
\usepackage{algpseudocode} 
\usepackage{amsmath}
\usepackage{amsthm}
\usepackage{amssymb}
\usepackage{mathpazo}
\usepackage{array}

\theoremstyle{definition}
\newtheorem{defi}{Definition}[section]
\newtheorem{rema}{Remark}[section]

\newtheorem{conj}{Conjecture}

\theoremstyle{plain}
\newtheorem{theo}{Theorem}[section]

\newtheorem{prop}[theo]{Proposition}
\newtheorem{coro}[theo]{Corollary}

\numberwithin{equation}{section}

\newcommand{\mN}{\mathbb{N}}
\newcommand{\mZ}{\mathbb{Z}}
\newcommand{\mP}{\mathbb{P}}

\newcommand{\mdot}{\!\cdot\!}
\newcommand{\mdiv}{\!\mid\!}
\newcommand{\ndiv}{\!\nmid\!}
\newcommand{\copr}{\!\perp\!}
\newcommand{\ncopr}{\!\not\perp\!}

\begin{document}

\title{Divisibility in paired progressions, Goldbach's conjecture,	and the infinitude of prime pairs}
\author{Mario Ziller and John F. Morack}
\date{}

\maketitle

\begin{abstract}

We investigate progressions in the set of pairs of integers $\mZ^2$ and define a\linebreak generalisation of the Jacobsthal function. For this function, we conjecture a specific\linebreak upper bound and prove that this bound would be a sufficient condition for the truth of the Goldbach conjecture, the infinitude of prime twins, and more general of prime pairs with a fixed even difference.\\
\end{abstract}

\section{Introduction}

Henceforth, we denote the set of integral numbers by $\mZ$ and the set of natural numbers, i.e. positive integers, by $\mN$. $\mP=\{p_i\mid i\in\mN\}$ is the set of prime numbers with $p_1=2$. As usual, we define the $n^{th}$ primorial number as the product of the first  $n$ primes: $p_n\#=\prod_{i=1}^n p_i\ , n\in\mN$. For similar objects of the specific context, we follow the notation of a recent paper \cite{Ziller_Morack_2016}.\\

We investigate the set of pairs of integers $\mZ^2=\mZ\times\mZ$ and first define divisibility of pairs with weak postulations. Afterwards, Jacobsthal's function \cite{Jacobsthal_1960_I, Erdos_1962} will be\linebreak generalised for the case of progressions of consecutive integer pairs using this concept of divisibility. In the subsequent sections, we demonstrate the relationship between\linebreak various unsolved problems in number theory, including Goldbach's conjecture and\linebreak the twin prime conjecture, with a specific bound of the generalised Jacobsthal\linebreak function.

\subsection*{Divisibility of integer pairs}	

Considering $\mZ^2$ as a canonical module over $\mZ$ with $(a,b)+(c,d)=(a+c,b+d)$ and $k\mdot(a,b)=(k\mdot a,k\mdot b)$ for all $a,b,c,d,k\in\mZ$, would imply strong requirements of divisibility: $k\mdiv (a,b)$ if and only if there exists $(c,d)\in\mZ^2$ with $(a,b)=k\mdot(c,d)$. This definition is equivalent to: $k\mdiv (a,b)$ if and only if $k\mdiv a$ and $k\mdiv b$. In this paper instead, we only use the weak divisibility as follows.

\begin{defi} {\itshape Divisor of a pair.}\\
An integer $k\in\mZ$ is a divisor of a pair $(a,b)\in\mZ^2$ if $k\mdiv a$ or $k\mdiv b$. We simply write $k\mdiv(a,b)$.
\end{defi}

\begin{rema}
\quad$k\mdiv(a,b)\Rightarrow k\mdiv a\mdot b$.\\
We remark that $k\mdiv a\mdot b$ does not necessarily mean $k\mdiv(a,b)$. If $k=k_1\mdot k_2$, $k_1\mdiv a$, and $k_2\mdiv b$ then $k_1\ndiv b$  and $k_2\ndiv a$ may pertain.\\
But prime numbers, on the other hand,  may be characterised by this condition:\\
\quad$p\in\mP\quad\Leftrightarrow\quad (\ p\mdiv a\mdot b\Rightarrow p\mdiv(a,b)\ )$.
\end{rema}

As a consequence of the last definition we declare coprimeness accordingly.

\begin{defi} {\itshape Coprime.}\\
An integer $k\in\mZ$ and a pair $(a,b)\in\mZ^2$ are coprime, and we write $k\copr(a,b)$, if $k\copr a$ and $k\copr b$.
\end{defi}

\begin{rema}
$k\copr(a,b)\Rightarrow k\ndiv a\mdot b$.\\
As a conclusion, an integer $k\in\mZ$ and a pair $(a,b)\in\mZ^2$ are not coprime, and we write $k\ncopr(a,b)$, if $k\ncopr a$ or $k\ncopr b$. In other words, there exists an $k^*\in\mZ$ with $k^*\mdiv k$ and $k^*\mdiv (a,b)$. Again, $k$ itself is not necessarily a divisor of $(a,b)$ here.
\end{rema}

\subsection*{Jacobsthal function}	

The ordinary Jacobsthal function $j(n)$ is known to be the smallest natural number $m$, such that every sequence of $m$ consecutive integers contains at least one integer coprime to $n$ \cite{Jacobsthal_1960_I, Erdos_1962}.

\begin{defi} \label{Jacobsthal} {\itshape Jacobsthal function.}\\
For $n\in\mN$, the Jacobsthal function $j(n)$ is defined as
$$j(n)=\min\ \{m\in\mN\mid\forall\ a\in\mZ\ \exists\ q\in\{1,\dots,m\}:a+q\copr n\}.$$
\end{defi}

\begin{rema}
This definition is equivalent to the formulation that $j(n)$ is the greatest difference $m$ between two terms in the sequence of integers which are coprime to $n$.
\begin{equation*} \begin{split}
j(n)=\max\ \{m\in\mN\mid\ &\exists\ a\in\mZ\ :a\copr n\land a+m\copr n\ \land\\
&\forall\ q\in\{1,\dots,m-1\}:a+q\ncopr n\}.
\end{split} \end{equation*}
In other words, $(j(n)-1)$ is the greatest length $m^*=m-1$ of a sequence of consecutive integers which are not coprime to $n$.
\end{rema}

An analogous function will be defined below for sequences of integer pairs in place of consecutive integers.

\subsection*{Paired progressions}	

Progressions of consecutive pairs in $\mZ^2$ can be defined in an intuitive way. There is a natural duality between a progression of consecutive pairs of integers and a pair of progressions of consecutive integers.

\begin{defi} {\itshape Consecutive pairs.}\\
\mbox{$(a_1,b_1), (a_2,b_2)\in\mZ^2$} are called consecutive pairs if \mbox{$a_2=a_1+1$} and \mbox{$b_2=b_1+1$}.\linebreak $(a_2,b_2)$ is called successor of $(a_1,b_1)$.
\end{defi}

\begin{defi} {\itshape Paired progression.}\\
An ordered  sequence of consecutive pairs $\{(a_i,b_i)\in\mZ^2\}_{i=1,\dots,k}$ is called a paired\linebreak progression if $(a_{i+1},b_{i+1})$ is a successor of $(a_i,b_i)$ for $i=1,\dots,k-1$.
\end{defi}

\begin{rema}
For every paired progression $\{(a_i,b_i)\in\mZ^2\}_{i=1,\dots,k}$ as defined above, there exists a pair $(a,b)\in\mZ^2$ with $(a_i,b_i)=(a+i,b+i)$ for $i=1,\dots,k$. We use the notation $\langle a,b\rangle_k$ for this paired progression and point to the fact that $(a,b)$ itself is not member of the progression.\\
\end{rema}

\section{Generalised Jacobsthal function}

We now generalise Jacobsthal's function on successive pairs of integers as its canonical extension to paired progressions and apply weak divisibility at it. 

\begin{defi} {\itshape Paired Jacobsthal function.}\label{j_2}\\
Let $n$ be a natural number. The paired Jacobsthal function $j_2(n)$ is defined to be the smallest natural number $m$, such that every paired progression $\langle a,b\rangle_m$ of length $m$ with an even difference of its pair elements contains at least one pair coprime to $n$.\\
\begin{equation*} \begin{split}
j_2(n)=\min\ \{m\in\mN\mid\ &\forall\ (a,b)\in\mZ^2\ with\ 2\mdiv (b-a):\\
&\exists\ (x,y)\in\langle a,b\rangle_m:n\copr (x,y)\},\\
or \quad j_2(n)=\min\ \{m\in\mN\mid\ &\forall\ (a,b)\in\mZ^2\ with\ 2\mdiv (b-a):\\
&\exists\ q\in\{1,\dots,m\}:n\copr (a+q,b+q)\}.
\end{split} \end{equation*}
\end{defi}

\begin{rema}
In the particular case of an odd difference of the pair elements, there is no pair coprime to an even $n$ because either the first or the second element of the pair would be even. This trivial case must be excluded. Otherwise $j_2(n)$ would not be defined for even $n$.
\end{rema}

This definition is equivalent to the formulation that the paired Jacobsthal function $j_2(n)$ is the greatest difference $m$ between two pairs in a sequence of consecutive pairs with an even difference of its pair elements, which are coprime to $n$.
\begin{equation*} \begin{split}
j_2(n)=\max\ \{m\in\mN\mid\ &\exists\ (a,b)\in\mZ^2\ with\ 2\mdiv (b-a):\\
&n\copr (a,b)\ \land\ n\copr (a+m,b+m)\ \land\\
&\forall\ q\in\{1,\dots,m-1\}:n\ncopr (a+q,b+q)\}.
\end{split} \end{equation*}

In other words, $(j_2(n)-1)$ is the greatest length $m^*=m-1$ of  a paired progression\linebreak $\langle a,b\rangle_{m^*}$ with an even difference of its pair elements where no pair is coprime to $n$.\\

\begin{rema} \label{properties}
The following statements are elementary consequences of the definition \ref{j_2} of the paired Jacobsthal function and describe some interesting properties of it. Equivalent properties are known for the common Jacobsthal function for integer\linebreak sequences \cite{Jacobsthal_1960_I,Ziller_Morack_2016}.

Product.\\ 
\hspace*{ 1cm}$\forall\ n1,n2 \in\mN:j_2(n1\mdot n2)\ge j_2(n1) \land j_2(n1\mdot n2)\ge j_2(n2)$.

Coprime product.\\ 
\hspace*{ 1cm}$\forall\ n1,n2 \in\mN>1\mid n1\copr n2:j_2(n1\mdot n2)>j_2(n1)\land j_2(n1\mdot n2)>j_2(n2)$.

Greatest common divisor.\\ 
\hspace*{ 1cm}$\forall\ n1,n2 \in\mN:j_2(gcd(n1,n2))\le j_2(n1) \land j_2(gcd(n1,n2))\le j_2(n2)$.

Prime power.\\ 
\hspace*{ 1cm}$\forall\ n,k \in\mN\ \forall\ p \in\mP:j_2(p^k\mdot n)=j_2(p\mdot n)$.

Prime separation.\\ 
\hspace*{ 1cm}$\forall\ n,n^*,k \in\mN\ \forall\ p \in\mP\mid n=p^k\mdot n^*,p\copr n^*:j_2(n)=j_2(p\mdot n^*)$.
\end{rema}

The last remark implies that the entire paired Jacobsthal function is also determined by its values for products of distinct primes. The particular case of primorial numbers is therefore most interesting because the function values at these points contain the relevant information for constructing general upper bounds. The paired Jacobsthal function of primorial numbers $h_2(n)$ is therefore defined as the smallest natural number $m$, such that every paired progression $\langle a,b\rangle_m$ of length $m$ with $2\mdiv (b-a)$ contains at least
one pair coprime to  the product of the first $n$ primes.\\

\begin{defi} {\itshape Primorial paired Jacobsthal function.} \label{h_2}\\
For $n\in\mN$, the primorial paired Jacobsthal function $h_2(n)$ is defined as
$$h_2(n)=j_2(p_n\#).$$
\end{defi}

In other words, $h_2(n)$ is the smallest length $m$ of a paired progression $\langle a,b\rangle_m$ with $2\mdiv (b-a)$ containing at least one pair coprime to all of the first $n$ primes.\\
\begin{equation*} \begin{split}
h_2(n)=\min\ \{m\in\mN\mid\ &\forall\ (a,b)\in\mZ^2\ with\ 2\mdiv (b-a):\\
&\exists\ (x,y)\in\langle a,b\rangle_m\ \forall\ i\in\{1,\dots,n\}:p_i\copr (x,y)\}.\\
\end{split} \end{equation*}\\

\section{Relationship to some open problems}

\subsection*{Goldbach conjecture}	

Goldbach formulated his original conjecture in a letter to Euler \cite{Goldbach_1742} dated June 7, 1742 \cite{Kronecker_1901, Dickson_1919_1}. It has become one of the most famous unsolved mathematical problems. For our considerations, we investigate the best known and sometimes called strong or binary variant of the problem. 

\begin{conj} {\itshape Goldbach conjecture.}\\
Every even natural number $2\mdot n$ with $n\in\mN>1$ can be expressed as the sum of two primes $q_1,q_2\in\mP$. 
$$\forall\ n\in\mN>1\ \exists\ q_1,q_2\in\mP:q_1+q_2=2\mdot n.$$
\end{conj}

We like to tighten this conjecture slightly and formulate a more constructive\linebreak assertion.

\begin{coro} \label{t_gold}
Let $n\in\mN\ge 6$, and $k_n$ denotes the index of the uniquely determined prime $p_{k_n}\in\mP$ with $p_{k_n}^2+p_{k_n}\le 2\mdot n<p_{k_n+1}^2+p_{k_n+1}$.

The Goldbach conjecture holds if for all $n\ge 6$ there exist two primes $q_1,q_2\in\mP$ with the conditions $p_{k_n}<q_1<p_{k_n}^2$ and $q_1+q_2=2\mdot n$.
\end{coro}

\begin{proof}
This is indeed a tightening because potential prime pairs with $q_1\le p_{k_n}$ and therefore $q_2\ge p_{k_n}^2$ are not considered. The smallest example therefor is $n=7$, $p_{k_n}=3$, $3+11=14$.

The assertion of the corollary is obvious for $n\ge 6$. The examples $4=2+2$, \linebreak $6=3+3,\ 8=3+5,\ \text{and}\ 10=3+7$ complete the proof.
\end{proof}

We point to a significant relation with the primorial paired Jacobsthal function\linebreak defined in the previous section. The tightened Goldbach conjecture holds if this\linebreak function is bounded in a specific way.

\begin{prop} \label{bound_t_gold}
Let $n\in\mN\ge 6$, and $k_n$ denotes the index of the prime $p_{k_n}\in\mP$ with $p_{k_n}^2+p_{k_n}\le 2\mdot n<p_{k_n+1}^2+p_{k_n+1}$.

If $h_2(k)<p_k^2-p_k$ holds for all $k\in\mN\ge 3$ then for all $n\ge 6$, there exist two primes $q_1,q_2\in\mP$ with the conditions $p_{k_n}<q_1<p_{k_n}^2$ and $q_1+q_2=2\mdot n$.
\begin{equation*} \begin{split}
&(\ \forall\ k\in\mN\ge 3:h_2(k)<p_k^2-p_k\ )\\
&\quad\Rightarrow\quad
(\ \forall\ n\in\mN\ge 6\ \exists\ q_1,q_2\in\mP:p_{k_n}<q_1<p_{k_n}^2\land q_1+q_2=2\mdot n\ ).
\end{split} \end{equation*}
\end{prop}

\begin{proof}
For $6\le n <30$, we get $k_n=2$, $p_{k_n}=3$, and $3<q_1<9$. The following examples fulfil the requirements:
$12=5+7,\ 14=7+7,\ 16=5+11,\ 18=5+13,\linebreak 20=7+13,\ 22=5+17,\ 24=5+19,\ 26=7+19,\ 28=5+23$.

Given now $n\ge 30$ and the appropriate $k_n$ as above. Then $k_n\ge 3$. Let $q_1,q_2\in\mP$ be a pair of primes with $p_{k_n}<q_1<p_{k_n}^2$ and $q_1+q_2=2\mdot n$. By assumption, we get the following inequations.

$q_1>p_{k_n}$, and

$q_1<p_{k_n}^2<p_{k_n+1}^2$.

$q_2=2\mdot n-q_1\ge p_{k_n}^2+p_{k_n}-q_1>p_{k_n}^2+p_{k_n}-p_{k_n}^2=p_{k_n}$, and

$q_2=2\mdot n-q_1<p_{k_n+1}^2+p_{k_n+1}-q_1\le p_{k_n+1}^2+p_{k_n+1}-p_{k_n+1}=p_{k_n+1}^2$

\qquad because $q_1$ is prime and therefore $q_1\ge p_{k_n+1}>p_{k_n}$.\\

For the primality of $q_1$ and $q_2$ with $p_{k_n}<q_1,q_2\in\mN<p_{k_n+1}^2$ and $q_1+q_2=2\mdot n$, it is necessary and sufficient  that $q_1\copr p_i$ and $q_2=2\mdot n-q_1\copr p_i$ for all $i=1,\dots, k_n$. Furthermore, $q_2\copr p_i$ if and only if $-q_2=q_1-2\mdot n\copr p_i$.

We now consider the paired progression $\langle p_{k_n},p_{k_n}-2\mdot n\rangle_{p_{k_n}^2-p_{k_n}-1}$ with
an obviously even difference of its corresponding pair elements. According to the assumption $h_2(k)<p_k^2-p_k$ for all $k\in\mN\ge 3$, we get $h_2(k_n)\le p_{k_n}^2-p_{k_n}-1$, and every paired\linebreak progression of length $p_{k_n}^2-p_{k_n}-1$ contains a pair coprime to  all $p_i,\ i=1,\dots,k_n$. And so it does for $\langle p_{k_n},p_{k_n}-2\mdot n\rangle_{p_{k_n}^2-p_{k_n}-1}$ and at least one of its pairs fulfills the requirements of the proposition.
\end{proof}

\subsection*{Twin prime conjecture}	

The twin prime conjecture is another best-known conjecture in number theory. It asserts that there are infinitely many prime twins, i.e. two primes with the difference 2. The true origin of this assertion is unrecorded.  Euclid proved the infinitude of primes. Sometimes, the twin prime conjecture was ascribed to Euclid therefore but with no reference. It has become common to regard the relation to the conjecture of de Polyniac \cite{de_Polignac_1849} and consider this as the origin of the twin prime conjecture, too.

The conjecture of de Polyniac \cite{Dickson_1919_1, Prinz_2012} asserts that every even natural number can be written in infinitely many ways as the difference of two consecutive primes. The twin prime conjecture is indeed included in this assertion when the even number is chosen to be 2.

\begin{conj} {\itshape Twin prime conjecture.}\\
There exist infinitely many pairs of primes $q_1,q_2\in\mP$ with the difference 2.
$$|\{q_1,q_2\in\mP\mid q_2-q_1=2\}|=\infty.$$
\end{conj}

\subsection*{Prime pairs conjecture}	

The question of the infinitude of prime pairs with a fixed even difference is also a\linebreak generalisation of the twin prime conjecture and nevertheless a weaker form of\linebreak de Polyniac's conjecture  \cite{Kronecker_1901, Dickson_1919_1, Prinz_2012}. We investigate this general conjecture because it is closely related to the paired Jacobsthal function in a similar way as we proved for the Goldbach conjecture.

\begin{conj} {\itshape Prime pairs conjecture.}\\
For every even natural number $d=2\mdot n$ with $n\in\mN$, there exist infinitely many pairs of primes $q_1,q_2\in\mP$ the difference of which is $d$.
$$\forall\ n\in\mN:|\{q_1,q_2\in\mP\mid q_2-q_1=2\mdot n\}|=\infty.$$
\end{conj}

\begin{coro}
If the prime pairs conjecture holds, so does the twin prime conjecture.
\end{coro}

\begin{proof}
The twin prime conjecture represents the specific case $n=1$ of the prime pairs conjecture.
\end{proof}

\begin{rema}
For $n=2$, every prime pair $q_1,q_2\in\mP$ with $q_1>3$ and $q_2-q_1=2\mdot n$ must be a pair of consecutive primes because $(q_1+q_2)/2$ is divisible by 3. The conjecture of de Polyniac and the prime pairs conjecture are even equivalent for $n\le 2$, therefore.
\end{rema}

\pagebreak

We tighten the prime pairs conjecture in a similar, constructive way as we did for Goldbach's conjecture.

\begin{coro} \label{t_pairs}
The prime pairs conjecture holds if for every natural number $n\in\mN$ and every prime $p\in\mP$ with $p>2\mdot n$, there exists a pair of primes $q_1,q_2\in\mP$ with the conditions
$p<q_1<p^2$ and $q_2-q_1=2\mdot n$.
\end{coro}

\begin{proof}
Given $n$, we choose an arbitrary $p\in\mP>2\mdot n$. By assumption, there exists a first pair $q_1,q_2\in\mP$ with $q_2-q_1=2\mdot n$ and $q_1>p>2\mdot n$.

Given any prime pair $q_i,q_j\in\mP$ with $q_j-q_i=2\mdot n$ and $q_i>2\mdot n$, we choose another $p\in\mP>q_i>2\mdot n$. Then there exists a pair $q_{i+1},q_{j+1}\in\mP$ with $q_{j+1}-q_{i+1}=2\mdot n$ and $q_{i+1}>p>q_{i}$.

Infinitude follows by induction.
\end{proof}
\bigskip

If the same upper bound of the primorial paired Jacobsthal function, which was related to the tightened Goldbach conjecture in proposition \ref{bound_t_gold}, holds then the above tightened prime pairs conjecture holds as well.

\begin{prop}\label{bound_t_pairs}
If $h_2(k)<p_k^2-p_k$ holds for all $k\in\mN\ge 3$ then for every natural number $n\in\mN$ and every prime $p\in\mP$ with $p>2\mdot n$, there exists a pair of primes $q_1,q_2\in\mP$ with the conditions $p<q_1<p^2$ and $q_2-q_1=2\mdot n$.
\end{prop}

\begin{proof}
$p=3$ is the only prime $p=p_k>2\mdot n,\ n\in\mN$ with $k<3$. Then $n=1$.\linebreak The example $q_1=5,\ q_2=7$ fulfills the requirements for this case.

Given now $n\in\mN$, we require $k\ge 3$ and choose an arbitrary $p_k\in\mP>2\mdot n$. Let $q_1,q_2\in\mP$ be a pair of primes with $p_k<q_1<p_k^2$ and $q_2-q_1=2\mdot n$. By assumption, we get the following inequations.

$q_1>p_k$, and

$q_1<p_k^2<p_{k+1}^2$.

$q_2=2\mdot n+q_1>q_1>p_k$, and

$q_2=2\mdot n+q_1<p_k+q_1<p_k+p_k^2<1+2\mdot p_k+p_k^2=(1+p_k)^2<p_{k+1}^2$.\\

For the primality of $q_1$ and $q_2$ with $p_k<q_1,q_2\in\mN<p_{k+1}^2$ and $q_2-q_1=2\mdot n$, it is necessary and sufficient  that $q_1\copr p_i$ and $q_2=2\mdot n+q_1\copr p_i$ for all $i=1,\dots, k$. 

We now consider the paired progression $\langle p_k,p_k+2\mdot n\rangle_{p_k^2-p_k-1}$ with
an obviously even difference of its corresponding pair elements. According to the assumption, we get $h_2(k)\le p_k^2-p_k-1$, and every paired progression of length $p_k^2-p_k-1$ contains a pair coprime to  all $p_i,\ i=1,\dots,k$. And so it does for $\langle p_k,p_k+2\mdot n\rangle_{p_k^2-p_k-1}$ and at least one of its pairs fulfills the requirements of the proposition.
\end{proof}

\pagebreak

\section{Conclusions}

In the previous section, we alleged three new conjectures. Two of them are tightenings of assumptions well-known for a long time. These are the Goldbach conjecture and the infinitude of prime pairs with a fixed even difference. The third case concerns an upper\linebreak bound of the primorial paired Jacobsthal function which we defined beforehand.\\

Below, we provide explicit formulations of these new conjectures.\\\\

\begin{conj} {\itshape Tightened Goldbach conjecture.}\\
Let $n\in\mN\ge 6$, and $k_n$ denotes the index of the uniquely determined prime $p_{k_n}\in\mP$ with $p_{k_n}^2+p_{k_n}\le 2\mdot n<p_{k_n+1}^2+p_{k_n+1}$. Then for all $n\ge 6$, there exist two primes $q_1,q_2\in\mP$ with the conditions $p_{k_n}<q_1<p_{k_n}^2$ and $q_1+q_2=2\mdot n$.
\begin{equation*} \begin{split}
&\forall\ n\in\mN\ge 6,\ p_{k_n}^2+p_{k_n}\le 2\mdot n<p_{k_n+1}^2+p_{k_n+1}\\
&\exists\ q_1,q_2\in\mP:p_{k_n}<q_1<p_{k_n}^2\ \land\ q_1+q_2=2\mdot n.
\end{split} \end{equation*}
\end{conj}

\bigskip

This conjecture was verified for all $n$ with $12\le 2\mdot n \le 10^8$. In the ancillary file \dq goldbach.lis\dq, we exemplarily provide an exhaustive list of all of the always smallest corresponding pairs $(q_1,q_2)$ with $q_1+q_2=2\mdot n$ for $12\le 2\mdot n \le 500000$, respectively.\\\\

\begin{conj} {\itshape Tightened prime pairs conjecture.}\\
For every natural number $n\in\mN$ and every prime $p\in\mP$ with $p>2\mdot n$, there exists\linebreak a pair of primes $q_1,q_2\in\mP$ with the conditions
$p<q_1<p^2$ and $q_2-q_1=2\mdot n$.
\begin{equation*} \begin{split}
&\forall\ n\in\mN\ \forall p\in\mP,\ p>2\mdot n\\
&\exists\ q_1,q_2\in\mP:p<q_1<p^2\ \land\ q_2-q_1=2\mdot n.
\end{split} \end{equation*}
\end{conj}

\bigskip

This conjecture was verified for all $n$ with $2\mdot n \le 10^4$ and all $3\le p<10^6$. In the ancillary file \dq pairs.lis\dq, we exemplarily provide an exhaustive list of all of the\linebreak always smallest corresponding pairs $(q_1,q_2)$ with $q_2-q_1=2\mdot n$ for $2\mdot n \le 100$ and\linebreak all $3\le p<25000$, respectively.\\\\

\begin{conj} {\itshape Upper bound of  the primorial paired Jacobsthal function.}\\
Let $n\in\mN\ge 3$. Then
$$h_2(n)<p_n^2-p_n.$$ 
\end{conj}

\pagebreak

The following scheme resumes and depicts the results of the previous sections. New conjectures are highlighted in grey.\\

\begin{figure}[!h]
  \centering
  \includegraphics[width=0.95\textwidth]{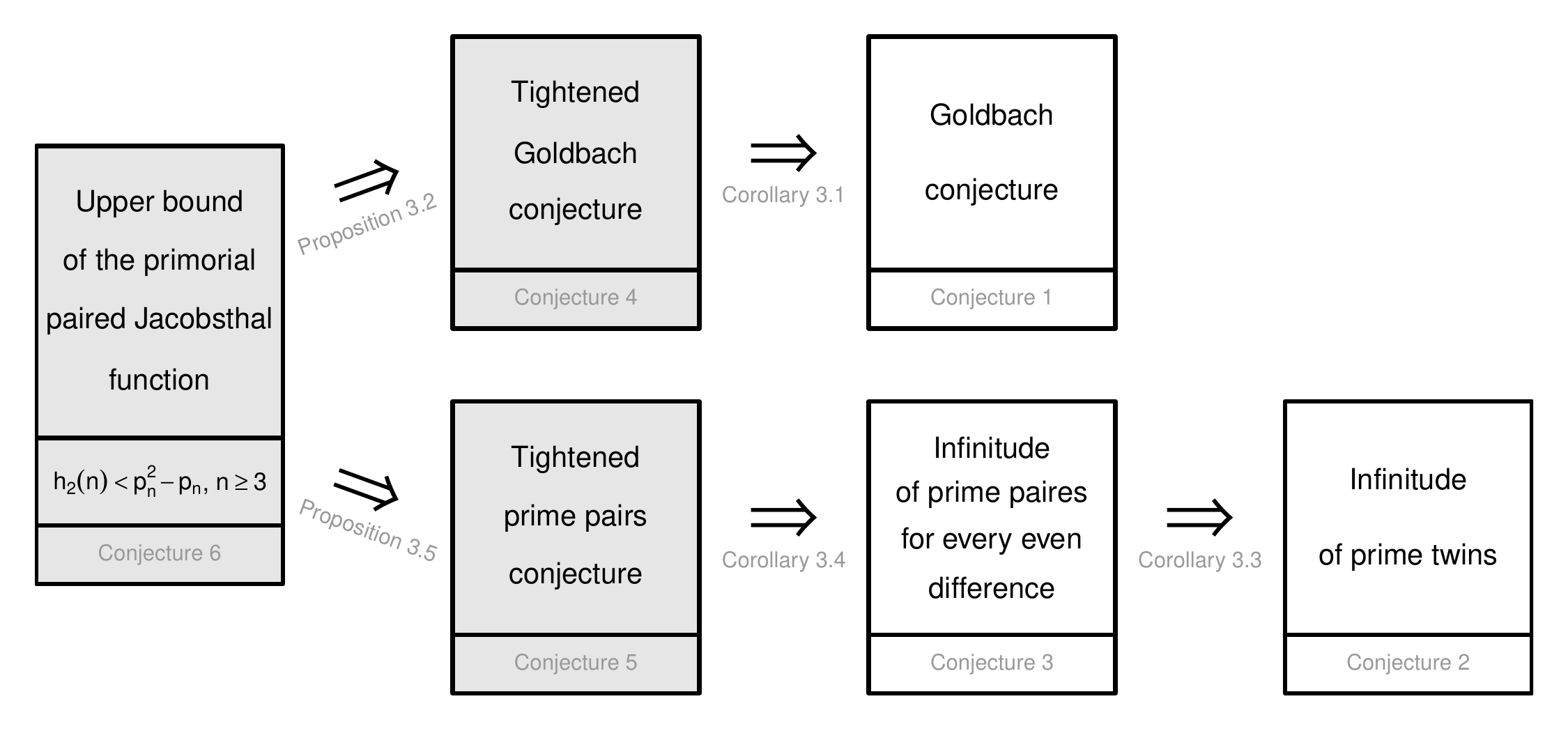}
  \caption{Scheme of proved inferences.}
\end{figure}

\bigskip
\bigskip
\bigskip

The main result of this paper is summarised in the following concluding theorem.\\

\begin{theo}
The conjectured upper bound of the primorial paired Jacobsthal function is sufficient for the truth of the Goldbach conjecture and of the infinitude of prime pairs for every even difference.
\end{theo}

\begin{proof}
This theorem follows from the propositions \ref{bound_t_gold} and \ref{bound_t_pairs} and the corollaries \ref{t_gold} and \ref{t_pairs}.
\end{proof}

\bigskip
\bigskip
\bigskip

\subsubsection*{Contact}
marioziller@arcor.de\\
axelmorack@live.com

\pagebreak


\bibliography{References}     

\begin{thebibliography}{10}

\bibitem{de_Polignac_1849}
Alphonse de~Polignac, \emph{{Six propositions arithmologiques d\'eduites du
  crible d'\'Eratosth\'ene}}, Nouvelles annales de math\'ematiques : journal
  des candidats aux \'ecoles polytechnique et normale \textbf{8} (1849),
  423--429.

\bibitem{Dickson_1919_1}
Leonard~Eugene Dickson, \emph{{History of the theory of numbers. Volume I.
  Divisibility and primality}}, Carnegie Institution of Washington, Washington,
  1919, 421-425.

\bibitem{Erdos_1962}
Paul Erd\"os, \emph{On the {Integers} {Relatively} {Prime} to $n$ and a
  {Number}-{Theoretic} {Function} {Considered} by {Jacobsthal}}, MATHEMATICA
  SCANDINAVICA \textbf{10} (1962), 163--170.

\bibitem{Goldbach_1742}
Christian Goldbach, \emph{{Letter to Euler}}, Correspondance math\'ematique et
  physique de quelques c\'el\`ebres g\'eom\`etres du {XVIII\'eme} si\'ecle
  (Paul~Heinrich Fuss, ed.), vol.~1, 1843, 125-129.

\bibitem{Jacobsthal_1960_I}
Ernst Jacobsthal, \emph{{\"Uber Sequenzen ganzer Zahlen, von denen keine zu n
  teilerfremd ist. I}}, D.K.N.V.S. Forhandlinger \textbf{33} (1960), no.~24,
  117--124.

\bibitem{Kronecker_1901}
Leopold Kronecker, \emph{{Vorlesungen \"uber Zahlentheorie}}, B. G. Teubner,
  Leipzig,\linebreak 1901, 68.

\bibitem{Prinz_2012}
Janos Pintz, \emph{On the difference of primes}, arXiv:1206.0149 [math.NT]
  (2012).

\bibitem{Ziller_Morack_2016}
Mario Ziller and John~F. Morack, \emph{{Algorithmic concepts for the computation\linebreak
  of Jacobsthal's function}}, arXiv:1611.03310 [math.NT] (2016).

\end{thebibliography}


\end{document}